\documentclass
{amsart}
\usepackage{amsmath}
\usepackage{amsthm}
\usepackage{amscd}
\usepackage{amsxtra}
\usepackage{amssymb}
\usepackage[all]{xy}
\usepackage{boxedminipage}
\usepackage{ulem}
\usepackage{stmaryrd}
\usepackage[dvipdfm]{graphicx,color}

\newtheorem{theorem}{Theorem}[section]

\newtheorem{corollary}[theorem]{Corollary}

\newtheorem*{H-S}{Theorem (Hovey-Sadofsky)}

\newtheorem{remark}{Remark}

\newtheorem*{acknowledgements}{Acknowledgements}

\theoremstyle{definition}

\newtheorem{def-pr}[theorem]{Definition-Proposition}
\newtheorem{def-th}[theorem]{Definition-Theorem}
\newtheorem{def-cor}[theorem]{Definition-Corollary}

\newcommand{\Hom}{ \operatorname{Hom} }

\newcommand{\Spec}{ \operatorname{Spec} }

\newcommand{\disc}{ \operatorname{disc} }

\newcommand{\lcm}{ \operatorname{lcm} }

\newcommand{\ord}{ \operatorname{ord} }

\newcommand{\Z}{{ \mathbb Z}}

\newcommand{\F}{{ \mathbb F}}

\newcommand{\N}{{ \mathbb N}}
	
\newcommand{\Ocal}{ {\mathcal{O}} }

\newcommand{\Ab}{ {\mathfrak{Ab}} }

\newcommand{\dotbox}{\hbox to 1em{\hss.\hss}}

\begin{document}

\title{Meromorphicity of some deformed multivariable zeta functions for $F_1$-schemes}
\author{Norihiko Minami}
\address{Omohi College, Nagoya Institute of Technology, Gokiso, Showa-ku, Nagoya 466-8555}
\email{nori@nitech.ac.jp}

\maketitle

\begin{abstract}
Motivated by 
recent work of Deitmar-Koyama-Kurokawa \cite{DKK}, Kurokawa-Ochiai \cite{KO1} \cite{KO2}, Connes-Consani \cite{CC}, 
and the author \cite{M1} \cite{M2}, 
we define multivariable deformed zeta functions of Hurwitz-Igusa type
\[
\zeta^{HI}(s_1,\ldots,s_r; a_1,\ldots, a_r; w; X)\quad ( \Re(a_i) > 0,\ 1\leq\forall i\leq r)
\]
for a Noetherian $\F_1$-scheme $X$ in the sense of \cite{CC}.  
Our zeta functions 
generalize both the zeta functions studied in \cite{DKK} \cite{KO2} and the log derivative
of the modified Soul\'e type zeta function \cite{CC}.
For these zeta functions, we give an explicit presentation using the Hurwitz zeta function,
from which, 
we show
\begin{enumerate}
\item When $w\in \N$,
$\zeta^{HI}(s_1,\ldots,s_r; a_1,\ldots, a_r; w; X)$ 
is a meromorphic function of $s_1,\ldots,s_r$. 
\item When $a_1 = \cdots = a_r = 1$,
$\zeta^{HI}(z_1,\ldots,z_r; 1,\ldots, 1; w; X)$
is a meromorphic function of $s_1,\ldots,s_r,w$. 
\end{enumerate}

Our explict presentation, when restricted to the log derivative
of the modified Soul\'e type zeta functions, gives us the following
expression of the (generalized) Soul\'e zeta function \cite{S} \cite{CC} $\zeta_X(s)$ and
the modified zeta function $\zeta^{\disc}_X(s)$ for a Noetherian $\F_1$-scheme $X$:
\begin{equation*} \label{The expression}
\begin{split}
&\quad \zeta_X(s) = e^{h_1(s)} \zeta^{\disc}_X(s)  \\
&
= e^{h_2(s) } \prod_{p\in X}
\left(
\left(  \prod_{j=0}^{n(p)} (s-j)^{ \left( - \binom{ n(p) }{j}   (-1)^{ n(p) - j } \right) } \right)^{   
{\color{blue}\mu\left( \prod_j \Z/ m_j(p)\Z \right) } }
\right) ,
\end{split}
\end{equation*}
where, for each $p\in X$, 
$
\Ocal_{X,p}^{\times} 
= \Z^{n(p)}\times {\color{blue}\prod_j \Z/ m_j(p)\Z},$ 
and $h_1(s)$ and $h_2(s)$ are some entire functions, and, for a finite abelian group
$A = \prod_{j=1}^k\left(\Z/n_j\Z\right)$,
\begin{equation*}
{\color{blue}{\mu(A)}} := \sum_{a \in A} \frac{1}{ |a| } =
\frac{1}{{
{\lcm(n_1,n_2,\ldots,n_k)}}}\sum_{l =1}^{{
{\lcm(n_1,n_2,\ldots,n_k)}}} \gcd(l,n_1) \gcd(l,n_2)\cdots \gcd(l, n_k) .
\end{equation*}

Since ${\color{blue}{\mu(A)}}$ is not necessarily a natural number, but a rational number in general, 
${\color{blue}\mu\left( \prod_j \Z/ m_j(p)\Z \right) }$ may be regarded as a local contribution
at $p\in X$ of the obstruction for the \lq\lq rationallty\rq\rq\ of $\zeta_X(s)$ and $\zeta^{\disc}_X(s)$.

\end{abstract}


As is well-known \cite{S} \cite{K} \cite{D}, the original Soul\'e's zeta function \cite{S} for $\F_1$-schemes
(whose definition we follow \cite{CC} in this paper) 
can not be defined for arbitrary $\F_1$-schemes .
To overcome this dificulty, two completely different kinds of more general zeta functios for $\F_1$-schemes 
have been proposed recently; first by Deitmar-Koyama-Kurokawa \cite{DKK} and Kurokawa-Ochiai \cite{KO2}, and, second by  
Connes-Consani \cite{CC}.  To review these two approaches, we mostly assume our $\F_1$-scheme $X$
to be a Noetherian $\F_1$-scheme in the sennse of \cite{CC}. In particular, $X$ has finitely many points, and, $\forall p\in X$,
the \lq\lq residue field\rq\rq\ $\Ocal_{X,p}^{\times}$ is a finitely generated abelian group.


We now recall the first kind of zeta functions due to \cite{DKK} \cite{KO2}:
Recently, Deitmar-Koyama-Kurokawa [DKK] defined {\it the absolue Igusa zeta function} $\zeta^I(s, X)$ for 
a Noetherian $\F_1$-scheme $X$ by
\begin{equation} \label{F1zeta}
\zeta^I(s, X) =  \sum_{m=1}^{\infty} \# X( \F_{1^m} ) m^{-s} = \sum_{m=1}^{\infty} \Big| \Hom\left( \Spec (\F_{1^m}), X \right) \Big| m^{-s},  
\end{equation}
generalizing the (single variable) group zeta function of Igusa type $\zeta_{\Ab}^I(s; A)$ for a finitely generated abelian group $A$
with the following properties:

\begin{itemize}
\item For the affine $\F_1$-scheme $\F_1[ A ]$ determined by a finitely generated abelian group $A,$
\begin{equation*}
\zeta^I(s, \Spec \F_1[A]) = \zeta_{\Ab}^I(s; A) := \sum_{m=1}^{\infty} \Big| \Hom_{\Ab}( A, \Z/mZ ) \Big| m^{-s} 
\end{equation*}
\item For $X$, 
a Noetherian $F_1$-scheme in the sense of Connes-Consani \cite{CC},
\begin{align*}
\zeta^I(s, X) &
=
\sum_{p\in X} \zeta^I\left(s, \Spec \F_1[ \Ocal_{X,p}^{\times} ] \right) \\
&= \sum_{p\in X} \zeta_{\Ab}^I \left( s; \Ocal_{X,p}^{\times} \right) 
= \sum_{p\in X}  \sum_{m=1}^{\infty} \Big| \Hom_{\Ab}(  \Ocal_{X,p}^{\times}, \Z/mZ ) \Big| m^{-s} 
\end{align*}
\end{itemize}
\cite{DKK} evaluated $\zeta_{\Ab}^I(s; A)$ for a general finitely generated abelian group $A$ in two different ways,
both of which are supposed to imply the meromorphicity of $\zeta_{\Ab}^I(s; A)$ with respect to $s$.
Although these computations in \cite{DKK} are erroneous, the mistakes were fixed in \cite{M2}.
In particular, this implies, $\zeta^I(s, X)$ is a meromorphic function of $s$. 

More recently, Kurokawa-Ochiai \cite{KO2} defined {\it the multivariable group zeta function of Igusa type}
$\zeta_{\Ab}^I(s_1,\ldots,s_r; A)$ for a finitely generated abelian group $A$:
\begin{equation} \label{mza}
\zeta_{\Ab}^I(s_1,\ldots,s_r; A) := \sum_{m_1, \ldots, m_r \geq 1} 
\Big| \Hom_{\Ab}( A, \Z/m_1\cdots m_r \Z ) \Big| m_1^{-s_1}\cdots m_r^{-s_r}
\end{equation}
and proved its meromorphicity
for the particular case when $A$ is the cyclic group.  Comparing \eqref{F1zeta} and \eqref{mza}, it is very natural to
define {\it the multivariable absolue Igusa zeta function} $\zeta^I(s_1,\ldots,s_r, X)$ for 
a Noetherian $\F_1$-scheme $X$ by
\begin{equation} \label{mF1zeta} 
\begin{split}
\zeta^I(s_1\ldots,s_r, X) &:=  \sum_{m_1,\cdots,m_r\geq 1} \# X( \F_{1^{m_1\cdots m_r}} ) m_1^{-s_1}\cdots m_r^{-s_r} \\
&= \sum_{m_1,\ldots, m_r\geq 1} \Big| \Hom\left( \Spec (\F_{1^{(m_1\cdots m_r)}}), X \right) \Big| m_1^{-s_1}\cdots m_r^{-s_r}   \\
&= \sum_{p\in X} \zeta_{\Ab}^I\left(s_1,\ldots, s_r; \Ocal_{X,p}^{\times}  \right)  \\
&:= \sum_{p\in X} \sum_{m_1,\cdots, m_r \geq 1}^{\infty} \Big| \Hom_{\Ab}( \Ocal_{X,p}^{\times}, \Z/m_1\cdots m_r\Z ) \Big|
m_1^{-s_1}\cdots m_r^{-s_r}
\end{split}
\end{equation}

Both \cite{DKK} and \cite{KO2} obtained some elementary number theoretical identities, using these 
zeta functions.  However, we \cite{M1}\cite{M2} gave purely elementary number theoretical proofs for these identities
and their generalizations from the view point of elementary probability theory.  In fact, 
from the view point of elementary probability theory,
it is vey natural to define and study 
{\it the {\color{red}{deformed}} multivariable zeta function of Igusa type}
$\zeta_{\Ab}^I(s_1, s_2, \ldots, s_r; {\color{red}{w}} ; A )$
for a finitely generated abelian group  $A,$ by
\begin{equation} \label{deformed_abelian}
\zeta_{\Ab}^I(s_1,\ldots, s_r; {\color{red}{w}} ; A ) :=
\sum_{m_1,\cdots, m_r \geq 1}^{\infty} \Big| \Hom_{\Ab}( A, \Z/m_1\cdots m_r\Z ) \Big|^{{\color{red}{w}}}
m_1^{-s_1}\cdots m_r^{-s_r} ,
\end{equation}
and {\it the {\color{red}{deformed}} multivariable zeta function of Igusa type}
$\zeta^I(s_1, s_2, \ldots, s_r; {\color{red}{w}} ; X )$
for  a Noetherian  $\F_1$-scheme $X,$ by
\begin{equation} \label{deformed}
\begin{split}
\zeta^I(s_1,\ldots, s_r; {\color{red}{w}} ; X ) &:=
\sum_{p\in X} \zeta_{\Ab}^I\left(s_1,\ldots, s_r; {\color{red}{w}} ;  \Ocal_{X,p}^{\times}  \right)  \\
&=\sum_{p\in X} \sum_{m_1,\cdots, m_r \geq 1}^{\infty} \Big| \Hom_{\Ab}( \Ocal_{X,p}^{\times}, \Z/m_1\cdots m_r\Z ) \Big|^{{\color{red}{w}}}
m_1^{-s_1}\cdots m_r^{-s_r} ,
\end{split}
\end{equation}
and study these properties.
Incidently, motivated by some study of Casimir energy of infinite symmetric groups \cite{KO1}, \cite{KO2} hinted
some {\it deformed} zeta function like \eqref{deformed_abelian}.
Although the analyticity of  
{the {\color{red}{deformed}} multivariable zeta function of Igusa type} is highly problematic, we shall prove in 
Corollary~\ref{deformedmeromorphicity} its meromorphicity with respect to $a_1, \ldots, a_r, {\color{red}{w}}$.

It should be pointed out that, for a finitely abelian group $A$, 
$\zeta_{\Ab}^I\left(s_1,\ldots, s_r; {\color{red}{w}} ;  A  \right)$ posseses the multivariable Euler product (c.f. \cite{BEL}):
\begin{equation} \label{EulerProduct}
\begin{split}
\zeta_{\Ab}^I\left(s_1,\ldots, s_r; {\color{red}{w}} ;  A  \right) &:=  
\sum_{m_1,\cdots, m_r \geq 1} \Big| \Hom_{\Ab}( A, \Z/m_1\cdots m_r\Z ) \Big|^{{\color{red}{w}}}
m_1^{-s_1}\cdots m_r^{-s_r}
\\
&=\prod_{p: \text{primes}} \sum_{k_1,\cdots, k_r \geq 0} \Big| \Hom_{\Ab}( A, \Z/p^{k_1+\cdots+ k_r}\Z ) \Big|^{{\color{red}{w}}}
p^{-(k_1s_1 +\cdots + k_rs_r)}
\end{split}
\end{equation}

We now recall the second kind of zeta functions due to \cite{CC}.
Generalizing the domain of definitions of 
\begin{equation} \label{NX}
N_X(n) := 
\# X( \F_{1^{n-1}} ) \quad ( n\in \Z_{\geqq 2} )
\end{equation}
to arbitrary $n\in [1, \infty)$ by the Nevanlinna theory,
Connes-Consani \cite{CC} generalized the Soul\'e zeta function $\zeta_X(s)$ to an arbitrary
Noetherian $\F_1$-scheme $X$ by
\begin{equation} \label{Soulezeta}
\frac{ \partial_s\zeta_X(s) }{ \zeta_X(s) } = - \int_{n\geq 1} N_X(n) n^{-s-1} du
\end{equation}

Furthermore, 
Connes-Consani \cite{CC} defined {\it the modified  Soul\'e 
zeta function} $\zeta^{\disc}_X(s)$
for a Noetherian $\F_1$-schme $X$ by
\begin{equation} \label{disczeta}
\begin{split}
\frac{ \partial_s\zeta^{\disc}_X(s) }{ \zeta^{\disc}_X(s) } &= - \sum_{n\geq 1} N_X(n) n^{-s-1}
= - \chi(X) 
- \sum_{m=1}^{\infty} \# X( \F_{1^m} ) (m+1)^{-s-1}  \\
&= - \chi(X) 
- \sum_{m=1}^{\infty} \Big| \Hom\left( \Spec (\F_{1^m}), X \right) \Big| (m+1)^{-s-1}, 
\end{split}
\end{equation}
where $\chi (X) := N_X(1)$ is called the Euler characteristic of $X$ \cite{S} \cite{K}.

Then, \cite{CC} related $\zeta_X(s)$ and $\zeta^{\disc}_X(s)$ by establishing the expression
\begin{equation} \label{equivalent}
\zeta_X(s) = e^{h(s)}\zeta^{\disc}_X(s)
\end{equation}
for some entire function $h(s)$.  
Thus, $\zeta^{\disc}_X(s)$ has the same singularities as $\zeta_X(s)$.

However, unlike \eqref{EulerProduct},
the Dirichlet series 
\begin{equation} \label{dirichlet-Hurwitz}
\begin{split}
\sum_{m=1}^{\infty} \# X( \F_{1^m} ) (m+1)^{-s-1}
&=  \sum_{m=1}^{\infty} \Big| \Hom\left( \Spec (\F_{1^m}), X \right) \Big| (m+1)^{-s-1}   \\
&=  \sum_{p\in X} \sum_{m=1}^{\infty} \Big| \Hom_{\Ab}\left( \Ocal_{X,p}^{\times}, \Z/m\Z \right) \Big| (m+1)^{-s-1}
\end{split}
\end{equation}
in \eqref{disczeta}
does not have an Euler product decomposition.  Therefore, we would also like to incorporate such a zeta function
without an Euler product decomposition into our study.
Especially, we are interested in
the essential ingredient in the characterization of
the modified zeta function of Soul\'e 
type $\zeta^{\disc}_X(s)$ :

\begin{equation} \label{mdirichlet}
 \sum_{m=1}^{\infty} \Big| \Hom\left( \Spec (\F_{1^m}), X \right) \Big| (m+1)^{-s-1}
=
\sum_{p\in X} \sum_{m=1}^{\infty} \Big| \Hom_{\Ab}\left( \Ocal_{X,p}^{\times}, \Z/m\Z \right) \Big| (m+1)^{-s-1}
\end{equation}

Now, with \eqref{deformed} and \eqref{mdirichlet} as our principal motivation, we define
{\it the deformed multivariable zeta function of Hurwitz-Igusa type}
$\zeta_{\Ab}^{HI}( s_1,\ldots, s_r; a_1, \ldots, a_r; {\color{red}{w}} ; A )$,
for a finitely generated abelian gropup $A$, by
\begin{equation} \label{deformedHIAb}
\begin{split}
&\qquad \zeta_{\Ab}^{HI}(s_1,\ldots, s_r; a_1, \ldots, a_r; {\color{red}{w}} ; A ) 
\\
&:= \sum_{m_1,\cdots, m_r \geq 1} \Big| \Hom_{\Ab}( A, \Z/m_1\cdots m_r\Z ) \Big|^{{\color{red}{w}}}
(m_1-1+a_1)^{-s_1}\cdots (m_r-1+a_r)^{-s_r}
\\
&\hspace{60mm} (\Re (a_i) > 0,\ 1\leq \forall i\leq r) ,
\end{split}
\end{equation}
and {\it the deformed multivariable zeta function of Hurwitz-Igusa type}
\newline
$\zeta^{HI}( s_1,\ldots, s_r; a_1, \ldots, a_r; {\color{red}{w}} ; X )$
for a Noetherian $\F_1$-scheme $X$, by
\begin{equation} \label{deformedHI}
\begin{split}
&\qquad \zeta^{HI}(s_1,\ldots, s_r; a_1, \ldots, a_r; {\color{red}{w}} ; X ) :=
\sum_{p\in X} \zeta_{\Ab}^{HI}\left(s_1,\ldots, s_r;  a_1, \ldots, a_r; {\color{red}{w}} ;  \Ocal_{X,p}^{\times}  \right)  \\
&=\sum_{p\in X} \sum_{m_1,\cdots, m_r \geq 1} \Big| \Hom_{\Ab}( \Ocal_{X,p}^{\times}, \Z/m_1\cdots m_r\Z ) \Big|^{{\color{red}{w}}}
(m_1-1+a_1)^{-s_1}\cdots (m_r-1+a_r)^{-s_r}
\\
&\hspace{60mm} (\Re (a_i) > 0,\ 1\leq \forall i\leq r) .
\end{split}
\end{equation}

Observe, for a Noetherian $\F_1$-scheme $X$,
\begin{subequations}
\begin{align}
\zeta^{HI}(s_1,\ldots, s_r; 1, \ldots, 1; {\color{red}{w}} ; X ) &=
\zeta^{I}(s_1,\ldots, s_r; {\color{red}{w}} ; X )   \label{deformedHI-I}   \\
\zeta^{HI}(s+1; 2; {\color{red}{1}}; X ) &=
\sum_{p\in X} \sum_{m=1}^{\infty} \Big| \Hom_{\Ab}\left( \Ocal_{X,p}^{\times}, \Z/m\Z \right) \Big| (m+1)^{-s-1}
\label{deformedHI-modified}
\end{align}
\end{subequations}

Of course, the analyticity of $\zeta^{HI}(s_1,\ldots, s_r; a_1, \ldots, a_r; {\color{red}{w}} ; X )$ is highly problematic,
and the main result of this paper investigates this issue:

\begin{theorem} \label{MT}
By use of the  Hurwitz zeta function
$\zeta(s, q ) := \sum_{n\geq 0} (n+q)^{-s}\ (\Re (s) >1, \Re(q) > 0)$, 
the multivariable deformed zeta function of Hurwitz-Igusa type for a Noetherian $\F_1$-scheme $X$ 
\[
\zeta^{HI}(s_1,\ldots,s_r; a_1,\ldots, a_r; w; X)\quad ( \Re(a_i) > 0,\ 1\leq\forall i\leq r)
\]
admits the following explicit presentation: 
\begin{equation} \label{explictMT}
\begin{split}
&\qquad \zeta^{HI}(s_1,\ldots,s_r; a_1,\ldots, a_r; w; X) 
\\
&= \sum_{p\in X} \zeta_{\Ab}^{HI}(s_1,\ldots,s_r; a_1,\ldots, a_r; w; \Ocal_{X,p} ) 
 \\
&= 
\begin{cases}
 \sum_{p\in X} \sum_{j_1,\ldots,j_r=0}^{n(p){{\color{red}{w}}}} 
\left[
\prod_{i=1}^r  
\left(  
 \binom{ n(p){{\color{red}{w}}} }{j_i}   (1-a_i)^{ n(p){{\color{red}{w}}} - j_i }
\right)
\right]
&  \\
\hspace{2mm}
\times \left[
\sum_{k_1,\ldots,k_r=1}^{l(p)} 
\Big| \Hom_{\Ab}( \Gamma_p, \Z/k_1\cdots k_r\Z ) \Big|^{{\color{red}{w}}}  
\prod_{i=1}^r 
\left( l(p)^{ - (s_i - j_i) } \zeta\left(  s_i-j_i, \frac{k_i-1+a_i}{l(p)} \right) \right) \right]  &  \\
\ &\hspace{-35mm}\text{if}\ w\in \N  \\
\sum_{p\in X}
\Big| \Hom_{\Ab}( \Gamma_p, \Z/k_1\cdots k_r\Z ) \Big|^{{\color{red}{w}}}  
\prod_{i=1}^r \left( l(p)^{ - (s_i - n(p){\color{red}{w}}) }
\zeta\left(  s_i-n(p){\color{red}{w}}, \frac{k_i}{l(p)} \right) \right)  & \\
 &\hspace{-35mm}\text{if}\ a_1 = \cdots = a_r = 1  ,\\
\end{cases}
\end{split}
\end{equation}
where, $\forall p\in X$, $\Ocal_{X,p}^{\times} = \Z^{n(p)}\times \Gamma_p$ with $\Gamma_p$
a finite abelian group such that $l(p) = \lcm \{ \ord (g) \mid g\in \Gamma \}$.
In particular,
$
\zeta^{HI}(s_1,\ldots,s_r; a_1,\ldots, a_r; w; X)\quad ( \Re(a_i) > 0,\ 1\leq\forall i\leq r)
$ 
enjoys the following meromorphicities:
\begin{enumerate}
\item When $w\in \N$,
$\zeta^{HI}(s_1,\ldots,s_r; a_1,\ldots, a_r; w; X)$ 
is a meromorphic function of $s_1,\ldots,s_r$. 
\item When $a_1 = \cdots = a_r = 1$,
$\zeta^{HI}(z_1,\ldots,z_r; 1,\ldots, 1; w; X)$
is a meromorphic function of $s_1,\ldots,s_r,w$. 
\end{enumerate}

\end{theorem}

\begin{proof} It suffices to study 
$\zeta_{\Ab}^{HI}(s_1,\ldots, s_r; a_1, \ldots, a_r; {\color{red}{w}} ; A )$
for a finitely generated abelian group $A= \Z^n \times \Gamma,$
where $\Gamma$ is a finite abelian group with
$l := \lcm \{ \ord (g) \mid g\in \Gamma \}$:
For $m_i \in \N$, set
\begin{equation} \label{G}
G(m_1,\ldots,m_r) := \Big| \Hom_{\Ab}( \Gamma, \Z/m_1\cdots m_r\Z ) \Big| ,
\end{equation}
and write $m_i = ln_i + k_i$ with $n_i\in \Z_{\geq 0}$ and $1\leq k_i \leq l$.
Then,
\begin{equation} \label{periodic}
G(m_1,\ldots,m_r) =  G(k_1,\ldots,k_r).
\end{equation}

Now, since
{
\begin{align*}
&\ \quad \Big| \Hom_{\Ab}( A, \Z/m_1\cdots m_r\Z ) \Big|^{{\color{red}{w}}}
(m_1-1+a_1)^{-s_1}\cdots (m_r-1+a_r)^{-s_r}
\\
&= \Big| \Hom_{\Ab}( \Gamma\times \Z^n, \Z/m_1\cdots m_r\Z ) \Big|^{{\color{red}{w}}}
(m_1-1+a_1)^{-s_1}\cdots (m_r-1+a_r)^{-s_r}
\\
&= G(m_1,\ldots,m_r) ^{{\color{red}{w}}} (m_1\cdots m_r)^{n{{\color{red}{w}}}} (m_1-1+a_1)^{-s_1}\cdots (m_r-1+a_r)^{-s_r}
\\
&= G(m_1,\ldots,m_r) ^{{\color{red}{w}}} \prod_{i=1}^r \left(  m_i^{n{{\color{red}{w}}}} (m_i-1+a_i)^{-s_i}   \right)
\\
&= 
\begin{cases}
G(m_1,\ldots,m_r) ^{{\color{red}{w}}} &  \\
\hspace{10mm}
\times \prod_{i=1}^r  \left(  (m_i-1+a_i) + (1-a_i) \right)^{n{{\color{red}{w}}}} (m_i-1+a_i)^{-s_i}
\ &\text{if}\ w\in \N  \\
G(m_1,\ldots,m_r) ^{{\color{red}{w}}} \prod_{i=1}^r \left(  m_i^{n{{\color{red}{w}}}} m_i^{-s_i}   \right)
\ &\text{if}\ a_1 = \cdots = a_r = 1\\
\end{cases}
\\
&= 
\begin{cases}
G(m_1,\ldots,m_r) ^{{\color{red}{w}}}  &  \\
\hspace{5mm}
\times \prod_{i=1}^r  
\left(  
\sum_{j_i=0}^{n{{\color{red}{w}}}} \binom{ n{{\color{red}{w}}} }{j_i}   (m_i-1+a_i)^{j_i} (1-a_i)^{ n{{\color{red}{w}}} - j_i }
\right)
 (m_i-1+a_i)^{-s_i}
\ &\text{if}\ w\in \N  \\
G(m_1,\ldots,m_r) ^{{\color{red}{w}}} \prod_{i=1}^r \left(  m_i^{-(s_i-n{{\color{red}{w}}} ) }   \right)
\ &\text{if}\ a_1 = \cdots = a_r = 1\\
\end{cases}
\\
&= 
\begin{cases}
\sum_{j_1,\ldots,j_r=0}^{n{{\color{red}{w}}}} \prod_{i=1}^r  
\left(  
 \binom{ n{{\color{red}{w}}} }{j_i}   (1-a_i)^{ n{{\color{red}{w}}} - j_i }
\right)
 &  \\
\hspace{10mm}
\times 
G(m_1,\ldots,m_r) ^{{\color{red}{w}}} 
\prod_{i=1}^r  
\left(  
 (m_i-1+a_i)^{-(s_i-j_i)}
\right)
\ &\text{if}\ w\in \N  \\
G(m_1,\ldots,m_r) ^{{\color{red}{w}}} \prod_{i=1}^r \left(  m_i^{-(s_i-n{{\color{red}{w}}} ) }   \right)
\ &\text{if}\ a_1 = \cdots = a_r = 1  ,\\
\end{cases}
\end{align*}
}
we see
{
\begin{equation} \label{reduction}
\begin{split}
&\qquad \zeta_{\Ab}^{HI}(s_1,\ldots, s_r; a_1, \ldots, a_r; {\color{red}{w}} ; A ) 
\\
&= \sum_{m_1,\cdots, m_r \geq 1} \Big| \Hom_{\Ab}( A, \Z/m_1\cdots m_r\Z ) \Big|^{{\color{red}{w}}}
(m_1-1+a_1)^{-s_1}\cdots (m_r-1+a_r)^{-s_r}
\\
&= 
\begin{cases}
 \sum_{j_1,\ldots,j_r=0}^{n{{\color{red}{w}}}} 
\prod_{i=1}^r  
\left(  
 \binom{ n{{\color{red}{w}}} }{j_i}   (1-a_i)^{ n{{\color{red}{w}}} - j_i }
\right)  & \\
\hspace{10mm}
\times \sum_{m_1,\cdots, m_r \geq 1}
G(m_1,\ldots,m_r) ^{{\color{red}{w}}} 
\prod_{i=1}^r  
\left(  
 (m_i-1+a_i)^{-(s_i-j_i)}
\right)
\ &\text{if}\ w\in \N  \\
\sum_{m_1,\cdots, m_r \geq 1} G(m_1,\ldots,m_r) ^{{\color{red}{w}}} \prod_{i=1}^r \left(  m_i^{-(s_i-n{{\color{red}{w}}} ) }   \right)
\ &\text{if}\ a_1 = \cdots = a_r = 1  .\\
\end{cases}
\end{split}
\end{equation}
}
Thus,
\begin{equation} \label{evaluation}
\begin{split}
&\quad\ \sum_{m_1,\cdots, m_r \geq 1}
G(m_1,\ldots,m_r) ^{{\color{red}{w}}} 
\prod_{i=1}^r  
\left(  
 (m_i-1+a_i)^{-(s_i-j_i)}
\right)
\\
&\overset{\eqref{periodic}}= \sum_{k_1,\ldots,k_r=1}^l \sum_{ n_1, \ldots, n_r \geq 0} G(k_1,\ldots,k_r) ^{{\color{red}{w}}} 
\prod_{i=1}^r  
\left(  
 (ln_i+k_i-1+a_i)^{-(s_i-j_i)}
\right)
\\
&= l^{ -\sum_{i=1}^r (s_i - j_i) } \sum_{k_1,\ldots,k_r=1}^l G(k_1,\ldots,k_r) ^{{\color{red}{w}}}  
\sum_{ n_1, \ldots, n_r \geq 0}
\prod_{i=1}^r 
\left(  
 \left(n_i+\frac{k_i-1+a_i}{l}\right)^{-(s_i-j_i)}
\right)
\\
&= l^{ -\sum_{i=1}^r (s_i - j_i) } \sum_{k_1,\ldots,k_r=1}^l G(k_1,\ldots,k_r) ^{{\color{red}{w}}}  
\prod_{i=1}^r 
\sum_{n_i\geq 0} 
\left(  
 \left( n_i+\frac{k_i-1+a_i}{l} \right)^{-(s_i-j_i)}
\right)
\\
&= l^{ -\sum_{i=1}^r (s_i - j_i) } \sum_{k_1,\ldots,k_r=1}^l G(k_1,\ldots,k_r) ^{{\color{red}{w}}}  
\prod_{i=1}^r 
\zeta\left(  s_i-j_i, \frac{k_i-1+a_i}{l} \right) 
\\
&= \sum_{k_1,\ldots,k_r=1}^l G(k_1,\ldots,k_r) ^{{\color{red}{w}}}  
\prod_{i=1}^r 
\left( l^{ - (s_i - j_i) } \zeta\left(  s_i-j_i, \frac{k_i-1+a_i}{l} \right) \right)
\end{split}
\end{equation}
and
\begin{equation} \label{evaluation2}
\begin{split}
&\quad \sum_{m_1,\cdots, m_r \geq 1} G(m_1,\ldots,m_r) ^{{\color{red}{w}}} \prod_{i=1}^r \left(  m_i^{-(s_i-n{{\color{red}{w}}} ) } \right)  \\
&\overset{\eqref{periodic}}= \sum_{k_1,\ldots,k_r=1}^l \sum_{ n_1, \ldots, n_r \geq 0} G(k_1,\ldots,k_r) ^{{\color{red}{w}}}  \prod_{i=1}^r\left(  (ln_i+k_i)^{-(s_i-n{{\color{red}{w}}} ) } \right)  \\
&= l^{ -\sum_{i=1}^r (s_i - n{\color{red}{w}}) } \sum_{k_1,\ldots,k_r=1}^l G(k_1,\ldots,k_r) ^{{\color{red}{w}}}  
\sum_{ n_1, \ldots, n_r \geq 0}
\prod_{i=1}^r 
\left(  
 \left(n_i+\frac{k_i}{l}\right)^{-(s_i-n{\color{red}{w}})}
\right)
\\
&= l^{ -\sum_{i=1}^r (s_i - n{\color{red}{w}}) } \sum_{k_1,\ldots,k_r=1}^l G(k_1,\ldots,k_r) ^{{\color{red}{w}}}  
\prod_{i=1}^r 
\sum_{ n_i \geq 0}
\left(  
 \left(n_i+\frac{k_i}{l}\right)^{-(s_i-n{\color{red}{w}})}
\right)
\\
&= l^{ -\sum_{i=1}^r (s_i - n{\color{red}{w}}) } \sum_{k_1,\ldots,k_r=1}^l G(k_1,\ldots,k_r) ^{{\color{red}{w}}}  
\prod_{i=1}^r 
\zeta\left(  s_i-n{\color{red}{w}}, \frac{k_i}{l} \right)
\\
&= \sum_{k_1,\ldots,k_r=1}^l G(k_1,\ldots,k_r) ^{{\color{red}{w}}}  
\prod_{i=1}^r \left( l^{ - (s_i - n{\color{red}{w}}) }
\zeta\left(  s_i-n{\color{red}{w}}, \frac{k_i}{l} \right) \right)
\end{split}
\end{equation}
are both meromorphic functions of $s_1,\ldots,s_r,w$.   



Then, substituting  \eqref{evaluation} and \eqref{evaluation2} into \eqref{reduction}, we find
{
\begin{equation} \label{reduction2}
\begin{split}
&\qquad \zeta_{\Ab}^{HI}(s_1,\ldots, s_r; a_1, \ldots, a_r; {\color{red}{w}} ; A ) 
\\
&= \sum_{m_1,\cdots, m_r \geq 1} \Big| \Hom_{\Ab}( A, \Z/m_1\cdots m_r\Z ) \Big|^{{\color{red}{w}}}
(m_1-1+a_1)^{-s_1}\cdots (m_r-1+a_r)^{-s_r}
\\
&= 
\begin{cases}
 \sum_{j_1,\ldots,j_r=0}^{n{{\color{red}{w}}}} 
\prod_{i=1}^r  
\left(  
 \binom{ n{{\color{red}{w}}} }{j_i}   (1-a_i)^{ n{{\color{red}{w}}} - j_i }
\right)
 l^{ -\sum_{i=1}^r (s_i - j_i) }   &  \\
\hspace{10mm}
\times
\sum_{k_1,\ldots,k_r=1}^l G(k_1,\ldots,k_r) ^{{\color{red}{w}}}  
\prod_{i=1}^r 
\zeta\left(  s_i-j_i, \frac{k_i-1+a_i}{l} \right)   \\
\ &\text{if}\ w\in \N  \\
l^{ -\sum_{i=1}^r (s_i - n{\color{red}{w}}) } \sum_{k_1,\ldots,k_r=1}^l G(k_1,\ldots,k_r) ^{{\color{red}{w}}}  
\prod_{i=1}^r 
\zeta\left(  s_i-n{\color{red}{w}}, \frac{k_i}{l} \right)
\ &\text{if}\ a_1 = \cdots = a_r = 1  .\\
\end{cases}
\\
&= 
\begin{cases}
 \sum_{j_1,\ldots,j_r=0}^{n{{\color{red}{w}}}} 
\left[
\prod_{i=1}^r  
\left(  
 \binom{ n{{\color{red}{w}}} }{j_i}   (1-a_i)^{ n{{\color{red}{w}}} - j_i }
\right)
\right]
&  \\
\hspace{10mm}
\times \left[
\sum_{k_1,\ldots,k_r=1}^l G(k_1,\ldots,k_r) ^{{\color{red}{w}}}  
\prod_{i=1}^r 
\left( l^{ - (s_i - j_i) } \zeta\left(  s_i-j_i, \frac{k_i-1+a_i}{l} \right) \right) \right] 
\ &\text{if}\ w\in \N  \\
\sum_{k_1,\ldots,k_r=1}^l G(k_1,\ldots,k_r) ^{{\color{red}{w}}}  
\prod_{i=1}^r \left( l^{ - (s_i - n{\color{red}{w}}) }
\zeta\left(  s_i-n{\color{red}{w}}, \frac{k_i}{l} \right) \right)
\ &\text{if}\ a_1 = \cdots = a_r = 1  .\\
\end{cases}
\end{split}
\end{equation}
}

Now the claim follows from these obsrevations and \eqref{reduction}.
\end{proof}


\begin{corollary} \label{deformedmeromorphicity}
The {\color{red}{deformed}} multivariable zeta function of Igusa type\newline
$\zeta^I(s_1, s_2, \ldots, s_r; {\color{red}{w}} ; X )$
for  a Noetherian  $\F_1$-scheme $X$:
\begin{equation} \label{deformed}
\begin{split}
\zeta^I(s_1,\ldots, s_r; {\color{red}{w}} ; X ) &:=
\sum_{p\in X} \zeta_{\Ab}^I\left(s_1,\ldots, s_r; {\color{red}{w}} ;  \Ocal_{X,p}^{\times}  \right)  \\
&:=\sum_{p\in X} \sum_{m_1,\cdots, m_r \geq 1}^{\infty} \Big| \Hom_{\Ab}( \Ocal_{X,p}^{\times}, \Z/m_1\cdots m_r\Z ) \Big|^{{\color{red}{w}}}
m_1^{-s_1}\cdots m_r^{-s_r}
\\
&=
\sum_{p\in X}
\Big| \Hom_{\Ab}( \Gamma_p, \Z/k_1\cdots k_r\Z ) \Big|^{{\color{red}{w}}}  
\prod_{i=1}^r \left( l(p)^{ - (s_i - n(p){\color{red}{w}}) }
\zeta\left(  s_i-n(p){\color{red}{w}}, \frac{k_i}{l(p)} \right) \right)
\end{split}
\end{equation}
is a meromorphic function of $s_1, \ldots, s_r, {\color{red}{w}}$.  
\end{corollary}
.
\begin{proof}  This immediately follows from Theorem~\ref{MT}~(2) and \eqref{deformedHI-I}.
\end{proof}

Now the special case of the {\color{red}{deformed}} modified zeta function of Soul\'e \cite{S} type 
$\zeta^{\disc}_X(s;{\color{red}{w}})$,
characterized by 
\begin{equation} \label{deformedSoule}
\frac{ \partial_s\zeta^{\disc}_X(s;{\color{red}{w}}) }{ \zeta^{\disc}_X(s;{\color{red}{w}}) } 
\equiv 
- \sum_{p\in X}
\sum_{m=1}^{\infty} \Big| \Hom_{\Ab}\left( \Ocal_{X,p}^{\times}, \Z/m\Z  \right) \Big|^{{\color{red}{w}}} (m+1)^{-s-1} ,
\end{equation}
modulo a constant, is of particular importance.

For this purpose, we recall from \cite{M2} the invariant $\mu(A)$  for a finite abelian group $A$ defined by
\begin{equation} \label{muA}
\mu (A) := \sum_{a \in A} \frac{1}{ |a| } ,
\end{equation}
where $|a|$ stands for the order of an element $a\in A$.
When
$A = \prod_{j=1}^k\left(\Z/n_j\Z\right)$,
the following evaluation is obtained in \cite{M2}:
\begin{equation} \label{mueval}
\begin{split}
\mu(A) &:= \sum_{a \in A} \frac{1}{ |a| } 
= \frac{1}{ |A| }\sum_{l=1}^{ |A| }  \Big|  \Hom_{\Ab}( A, \Z/l\Z ) \Big|  \\
&=
\frac{1}{{
{\lcm(n_1,n_2,\ldots,n_k)}}}\sum_{l =1}^{{
{\lcm(n_1,n_2,\ldots,n_k)}}}
\Big| \Hom_{\Ab}( A, \Z/l\Z ) \Big|  \\
&=
\frac{1}{{
{\lcm(n_1,n_2,\ldots,n_k)}}}\sum_{l =1}^{{
{\lcm(n_1,n_2,\ldots,n_k)}}} \gcd(l,n_1) \gcd(l,n_2)\cdots \gcd(l, n_k) ,
\end{split}
\end{equation}
whose last quantity was essentially first considered by \cite{DKK}.

When $A = \prod_{j=1}^k\left(\Z/n_j\Z\right)$ and ${\color{red}{w}}\in N$, 
$A^{\color{red}{w}} = \prod_{j=1}^k\left(\Z/n_j\Z\right)^{\color{red}{w}}$. Thus, applying \eqref{mueval} to the case 
$A^{\color{red}{w}}$, we obtain the following:
\begin{equation} \label{muevalw}
\mu(A^{\color{red}{w}}) = 
\frac{1}{ {\lcm(n_1,n_2,\ldots,n_k)} }\sum_{l=1}^{ {\lcm(n_1,n_2,\ldots,n_k)} }  
\Big|  \Hom_{\Ab}( A, \Z/l\Z ) \Big|^{\color{red}{w}} 
\end{equation}

Finally, we are able to state and prove our main result for 
the {\color{red}{deformed}} modified zeta function of Soul\'e \cite{S} type 
$\zeta^{\disc}_X(s;{\color{red}{w}})$:

\begin{theorem} \label{modified meromorphicity}
We have the following expression for the {\color{red}{deformed}} modified zeta function of Soul\'e \cite{S} type
for a Noetherian $\F_1$-scheme $X$
when ${\color{red}{w}}\in \N$:
\begin{equation} \label{The expression}
\begin{split}
&\quad \zeta^{\disc}_X(s;{\color{red}{w}})  \\
&= e^{h(s;{\color{red}{w}} ) } \prod_{p\in X}
\left(
\left(  \prod_{j=0}^{n(p)} (s-j)^{ \left( - \binom{ n(p){{\color{red}{w}}} }{j}   (-1)^{ n(p){{\color{red}{w}}} - j } \right) } \right)^{   
\mu( \Gamma_p^{ {\color{red}{w}} } )
}
\right) .
\end{split}
\end{equation}
Here, for all $p\in X$, 
\begin{equation} \label{local ring}
\Ocal_{X,p}^{\times} = \Z^{n(p)}\times \Gamma_{p},\quad
l(p) := \lcm \{ \ord ( g ) \mid g\in \Gamma_{p} \},
\end{equation}
with $\Gamma_p$ a finite abelian group;
and $h(s;{\color{red}{w}})$ is some entire function of $s$ depending upon ${\color{red}{w}}\in\N$.
\end{theorem}

\begin{proof}
Applying the preceeding results, we analyze the right hand side of \eqref{deformedSoule}:
\begin{equation} \label{SouleReduction}
\begin{split}
&\quad 
- \sum_{p\in X} \sum_{m=1}^{\infty} \Big| \Hom_{\Ab}\left( \Ocal_{X,p}^{\times}, \Z/m\Z \right) \Big|^{{\color{red}{w}}} (m+1)^{-s-1}  \\
&\overset{\eqref{deformedHIAb}}{=}  - \sum_{p\in X}  \zeta_{\Ab}^{HI}\left(s+1,2; {\color{red}{w}}; \Ocal_{X,p}^{\times} \right)   \\
&\overset{\eqref{reduction2}}{=}  - \sum_{p\in X} 
 \sum_{j=0}^{n(p){{\color{red}{w}}}} 
 \binom{ n(p){{\color{red}{w}}} }{j}   (-1)^{ n(p){{\color{red}{w}}} - j }
 l(p)^{ - (s+1 - j) } \sum_{k=1}^{l(p)} G_p(k) ^{{\color{red}{w}}}  
\zeta\left(  s+1-j, \frac{k+1}{l(p)} \right) , 
\end{split}
\end{equation}
where
\begin{equation} \label{Gp}
G_p(k) \overset{\eqref{G}}{=} \Big| \Hom_{\Ab}\left( \Gamma_p, \Z/k\Z \right) \Big| .
\end{equation}
Since the Hurwitz zeta $\zeta(s; q )$ only has a pole of residue $1$ at $s=1$, we see the only singularities of \eqref{SouleReduction} 
are poles at $s = j \in \cup_{p\in X} \{ 0, \cdots, n(p) \}$ with residue
\begin{equation} \label{residue}
\begin{split}
&\quad - \sum_{p\in X} 
 \sum_{j=0}^{n(p){{\color{red}{w}}}} 
 \binom{ n(p){{\color{red}{w}}} }{j}   (-1)^{ n(p){{\color{red}{w}}} - j }
 l(p)^{ - (1) } \sum_{k=1}^{l(p)} G(k) ^{{\color{red}{w}}}   \\
&=  \sum_{p\in X}  \sum_{j=0}^{n(p){{\color{red}{w}}}} 
\left(
- \binom{ n(p){{\color{red}{w}}} }{j}   (-1)^{ n(p){{\color{red}{w}}} - j } 
\frac{\sum_{k=1}^{l(p)}   G(k) ^{{\color{red}{w}}} }{  l(p) } 
\right)  \\
&\overset{\eqref{Gp}}{=}  \sum_{p\in X}  \sum_{j=0}^{n(p){{\color{red}{w}}}} 
\left(
- \binom{ n(p){{\color{red}{w}}} }{j}   (-1)^{ n(p){{\color{red}{w}}} - j } \right) 
\frac{\sum_{k=1}^{l(p)}   \Big| \Hom_{\Ab}\left( \Gamma_p, \Z/k\Z \right) \Big|^{{\color{red}{w}}} }{  l(p) } 
\\
&\overset{\eqref{muevalw}}{=} \sum_{p\in X}  \sum_{j=0}^{n(p){{\color{red}{w}}}} 
\left(
- \binom{ n(p){{\color{red}{w}}} }{j}   (-1)^{ n(p){{\color{red}{w}}} - j } \right) 
\mu( \Gamma_p^{{\color{red}{w}}} )
\end{split}
\end{equation}
Now the claim follows immediately.
\end{proof}







Further restricting to the case $w=1$, we arrive at the following observation,
which may be viewed as an eventual reconciliation of the two approaches of Deitmar-Koyama-Kurikawa \cite{DKK} 
and Connes-Consani \cite{CC}:

\begin{theorem}
the following
expression of the (generalized) Soul\'e zeta function $\zeta_X(s)$ \cite{S} \cite{CC} 
and the modified zeta function $\zeta^{\disc}_X(s)$ 
for a Noetherian $\F_1$-scheme $X$:
\begin{equation*} \label{The expression}
\begin{split}
&\quad\zeta_X(s) = e^{h_1(s)} \zeta^{\disc}_X(s)  \\
&
= e^{h_2(s) } \prod_{p\in X}
\left(
\left(  \prod_{j=0}^{n(p)} (s-j)^{ \left( - \binom{ n(p) }{j}   (-1)^{ n(p) - j } \right) } \right)^{   
{\color{blue}\mu\left( \prod_j \Z/ m_j(p)\Z \right) } }
\right) ,
\end{split}
\end{equation*}
where, for each $p\in X$, 
$
\Ocal_{X,p}^{\times} 
= \Z^{n(p)}\times {\color{blue}\prod_j \Z/ m_j(p)\Z},$ 
and $h_1(s), h_2(s)$ are some entire functions, and, furthermore, for a finite abelian group
$A = \prod_{j=1}^k\left(\Z/n_j\Z\right)$,
\begin{equation*}
{\color{blue}{\mu(A)}} := \sum_{a \in A} \frac{1}{ |a| } =
\frac{1}{{
{\lcm(n_1,n_2,\ldots,n_k)}}}\sum_{l =1}^{{
{\lcm(n_1,n_2,\ldots,n_k)}}} \gcd(l,n_1) \gcd(l,n_2)\cdots \gcd(l, n_k) .
\end{equation*}
\qed
\end{theorem}

\begin{remark}
Since ${\color{blue}{\mu(A)}}$ is not necessarily a natural number, but a rational number in general, 
${\color{blue}\mu\left( \prod_j \Z/ m_j(p)\Z \right) }$ may be regarded as a local contribution
at $p\in X$ of the obstruction for the \lq\lq rationality\rq\rq\ of $\zeta_X(s)$ and $\zeta^{\disc}_X(s)$.
\end{remark}

\begin{acknowledgements} {\rm
The basic idea of the results in this paper were obtained during the author's stay at JAMI2009, Johns Hopkins University, in March 2009,
and presented at NCGOA2009, Vanderbilt University, in May 2009, and at the Fall Meeting of the Mathematical Society of Japan
at Osaka University, in September 2009.  The author would like to express his gratitude to 
Katia Concani, Alain Connes, Jack Morava, Takashi Ono, Steve Wilson, and Guoliang Yu for their hospitalities.
The author also would like to express his gratitude to Nobushige Kurokawa for his work and encouragement.
}
\end{acknowledgements}

\end{document}